\documentclass{amsart}
\usepackage{url}

\usepackage{amsmath,amsfonts,amsthm,amssymb}
\usepackage[usenames]{color}
\definecolor{citegreen}{rgb}{0,0.6,0}
\definecolor{refred}{rgb}{0.8,0,0}
\usepackage[colorlinks, citecolor=citegreen, linkcolor=refred]{hyperref}

\newtheorem{thm}{Theorem}

\newtheorem{definition}[thm]{Definition}

\newtheorem{lemma}[thm]{Lemma}

\numberwithin{equation}{section} \numberwithin{thm}{section}
\newcommand{\R}{\mathbb{R}}
\newcommand{\N}{\mathbb{N}}

\newcommand{\Z}{\mathbb{Z}}

\renewcommand{\l}{\lambda}

\newcommand{\loc}{\mathrm{loc}}
\newcommand{\p}{\partial}

\allowdisplaybreaks

\begin{document}
\title[Conservation laws]{Conservation laws for even order elliptic systems in the critical dimension - a new approach}

\author[J.~H\"orter]{Jasmin H\"orter}
\address[J.~H\"orter]{Department of Mathematics\\ 
Karlsruhe Institute of Technology \\ 
76128 Karlsruhe\\ Germany}
\email{jasmin.hoerter@kit.edu}

\author[T.~Lamm]{Tobias Lamm}
\address[T.~Lamm]{Department of Mathematics\\ 
Karlsruhe Institute of Technology \\ 
76128 Karlsruhe\\ Germany}
\email{tobias.lamm@kit.edu}

\thanks{The authors acknowledge funding by the Deutsche Forschungsgemeinschaft (DFG, German Research Foundation) – 281869850 (RTG 2229)}

\date{\today}

\begin{abstract} 
We consider elliptic systems of order $2m$ in dimension $2m$ which are generalizations of extrinsic and intrinsic polyharmonic maps. We show the existence of a conservation law for these systems by using a small perturbation of Uhlenbeck's gauge fixing matrix.
\end{abstract}

\maketitle

\section{Introduction}
The regularity of critical points of geometric variational problems for maps between two Riemannian manifolds attracted a lot of attention over the last two decades. The most prominent example are  the harmonic maps which are critical points  $u\in W^{1,2}(M,N)$ of the Dirichlet energy
\[
E(u)=\frac12 \int_M |\nabla u|^2 \, dv_g,
\]
where $(M,g)$ and $(N,h)$ are two smooth and compact manifolds without boundary and $N$ is isometrically embedded into some euclidean space $\R^n$. They solve the elliptic system 
\[
-\Delta u =A(u)(\nabla u, \nabla u) ,
\]
where $A$ is the second fundamental form of the embedding $N\hookrightarrow \R^n$.  The Dirichlet energy is scaling invariant in dimension two, which is called the critical dimension, and it was shown by H\'elein \cite{Helein91} that weakly harmonic maps are smooth in this case.

This result was substantially extended by Rivi\`ere \cite{Riviere07} to more general elliptic systems of the form
\[
-\Delta u= \Omega \cdot \nabla u,
\]
where $\Omega\in L^2(B^2, so(n) \otimes \wedge^1(\R^2))$ and $B^2$ denotes the unit ball in $\R^2$. Rivi\`ere obtained the regularity of weak solutions as a consequence of a conservation law which he derived using the antisymmetry of $\Omega$. The key ingredient here was the use of the Uhlenbeck gauge fixing result \cite{Uhlenbeck82}, see Theorem \ref{Uhlenbeck for harm}. Note that the Euler-Lagrange equation of all quadratic and conformally invariant variational integrals satisfies an equation of the type above. We sketch a version of this result in section \ref{section second order case}.

The regularity result of H\'elein was than extended to the so-called weakly biharmonic maps in $\R^4$, i.e. critical points of the functional
\[
E_2(u)=\frac12 \int_M |\Delta u|^2 \, dv_g
\]
by Chang-Wang-Yang \cite{ChangWangYang} for spherical targets and by Wang \cite{Wang04} for general targets. Later, the second author and Rivi\`ere \cite{LammRiv} were able to show a conservation law for a suitable generalization of the biharmonic map equation in the spirit of the before mentioned paper of Rivi\`ere.  A modified version of this conservation law was later obtained by Struwe \cite{Struwe08}.

De Longueville and Gastel \cite{LongGas} recently extended this result to systems of order $2m$ in the critical dimension. The motivating example behind this system are the $m$-polyharmonic maps $u\in W^{m,2}(B^{2m},N)$, which are critical points of the functional 
\begin{align*}
E_m(u)=\frac12 \int_{B^{2m}} |\nabla^{m}u|^2\, dv_g.
\end{align*}
The Euler-Lagrange equation for $E_m$ was calculated by Angelsberg-Pumberger \cite{AngelsbergPumberger09} resp. Gastel-Scheven \cite{GastelScheven09}. In the latter paper the authors also showed the regularity for these critical points using H\'elein's moving frame technique.

In the following we consider systems of the from 
\begin{align}
\Delta^m u&= \sum_{k=0}^{m-1} \Delta^k \langle V_k, du\rangle + \sum_{k=0}^{m-2} \Delta^k \delta (w_k du).\label{msystem}
\end{align}
with  coefficient functions 
\begin{align*}
w_k&\in W^{2k+2-m,2}(B^{2m}, \R^{n\times n})\qquad\qquad\qquad\text{for }k\in\{0,...,m-2\},
\\
V_k&\in W^{2k+1-m,2}(B^{2m}, \R^{n\times n}\otimes \wedge^1 \R^{2m} )\qquad\text{for }k\in \{0,...,m-1\}, \text{ where}
\\
V_0&= d\eta +F,~~\eta\in W^{2-m,2}(B^{2m}, so(n)),~~ F\in W^{2-m, \frac{2m}{m+1},1}(B^{2m}, \R^{n\times n}\otimes \wedge^1 \R^{2m}).
\end{align*}
It was shown by De Longueville and Gastel \cite{LongGas} that $m$-polyharminic maps are solutions of a system of this type. Note that the definition and the basic properties of the negative Sobolev spaces arising in this equation are collected in section \ref{sec Lorentz}.

In our main Theorem \ref{existence thm} we establish a new conservation law for systems of the form \eqref{msystem}. The novelty here is that we use a small perturbation of the gauge fixing matrix $P$ in a suitable variant of the Uhlenbeck result, see Theorem \ref{Uhlenbeck thm}.

The paper is organized as follows. In section \ref{sec Lorentz} we recall some basic definitions and properties for negative Sobolev and Lorentz-Sobolev spaces and we show a suitable higher order generalization of the Wente Lemma. Moreover, we state and comment on our main Theorem.

In section \ref{section second order case} we review the second order case of our main result, a proof of which was  already sketched by Rivi\`ere in \cite{RivLecture}.

In section \ref{sec Proof} we finally show our main Theorem.

\section{Lorentz-Sobolev spaces and the main result}\label{sec Lorentz}		
In this section we start by recalling the definitions of the relevant function spaces we need in order to obtain the desired conservation law. Moreover, we show a preliminary result on a higher order version of the  famous Wente lemma \cite{Wente69} and we state our main result.

\subsection{Lorentz- and Lorentz-Sobolev spaces}

Important function spaces in our paper are the so called Lorentz spaces. They are interpolation spaces of the classical $L^p$-spaces and in the following we briefly collect a few properties of these spaces. For detailed proofs see for example \cite{LongGas,Grafakos,Helein,Hunt,Tartar,Ziemer}.
We start with a Lemma on the H\"older inequality for these functions 
\begin{lemma}[H\"older inequality]\label{Holder for Lorentz}
	Let $f\in L^{p_1,q_1}(\R^n)$ and $g\in L^{p_2,q_2}(\R^n)$ with $\frac{1}{p_1}+\frac{1}{p_2}=\frac{1}{p}, ~\frac{1}{q_1}+\frac{1}{q_2}=\frac{1}{q}$ and $p_1,p_2\in(1,\infty),~ q_1,q_2\in [1,\infty]$. Then  
	\begin{align*}
	||f g||_{L^{p,q}(\R^n)}\leq ||f||_{L^{p_1,q_1}(\R^n)}||g||_{L^{p_2,q_2}(\R^n)}.
	\end{align*} 	
\end{lemma}
Additionally, we also need the following estimates.
\begin{lemma}\label{embedding lorentz exponent}
	Let $f:\R^n\rightarrow \R$ be measurable.
	\begin{itemize}
		\item[1.] 	Let  $1<p\leq \infty$ and $1\leq q<Q\leq \infty$. Then we have 	
		\begin{align*}
		||f||_{L^{p,Q}(\R^n)}\leq c||f||_{L^{p,q}(\R^n)}.
		\end{align*}
		
		\item[2.] Let $1<p<P\leq \infty$, $1\leq q_1,q_2\leq \infty$ and let $\Omega\subset \R^n$ be bounded. Then we have 
		\begin{align*}
		||f||_{L^{p,q_1}(\Omega)}\leq c |\Omega|^{\frac{1}{p}-\frac{1}{P}} ||f||_{L^{P,q_2}(\Omega)}.
		\end{align*}
	\end{itemize}
\end{lemma}

Next we come to Lorentz-Sobolev spaces. If a function $f\in L^{p,q}(\R^n)$ has derivatives $D^j f\in L^{p,q}(\R^n)$ for all $1\leq j \leq k\in \N$, then  $f$ is an element of the so-called Lorentz-Sobolev space $W^{k,p,q}(\R^n)$. 

\begin{definition}
	Let $1<p\leq \infty,~1\leq q\leq \infty$ and $k\in \N$. Let $f\in L^{p,q}(\R^n)$ be $k$ times weakly differentiable and for all multiindices $\alpha\in \N_0^n$ with $|\alpha|\leq k$ let $\frac{\p^{|\alpha|}}{\p^{\alpha_1}x_1...\p^{\alpha_n}x_n} f\in L^{p,q}(\R^n)$. Then $f$ is an element of the Lorentz-Sobolev space  $ W^{k,p,q}(\R^n)$ with norm 
	\begin{align*}
	||f||_{W^{k,p,q}(\R^n)}:= \sum_{0\leq |\alpha|\leq k}\left\Vert \frac{\p^{|\alpha|}}{\p^{\alpha_1}x_1...\p^{\alpha_n}x_n} f\right\Vert_{L^{p,q}(\R^n)}.
	\end{align*} 
\end{definition}
We have a generalized Sobolev embedding theorem for these spaces.
\begin{lemma}\label{pos. LorentzSobolev Embedding}
	Let $k,n\in \N,~ 1<p<\frac{n}{k}$ and $1\leq q\leq \infty$. Then 
	\begin{align*}
	W^{k,p,q}(B^n)\hookrightarrow L^{p^*, q}(B^{n})
	\end{align*}
	for $\frac{1}{p^*}= \frac{1}{p}+\frac{k}{n}$ with the estimate
	\begin{align*}
	||f||_{L^{p^*,q}(B^n)}\leq c ||f|_{W^{k,p,q}(B^n)}\qquad\text{ for any $f\in W^{k,p,q}(B^n).$}
	\end{align*}
	\end{lemma}
Similar to Lemma \ref{Holder for Lorentz} we have a product estimate for Lorentz-Sobolev functions.
\begin{lemma}\label{Prodsatz Sobolev Lorentz}
	Let $ s,k\in \N ,~ p,p',q,q'\in \R$ with $1<p,p',q,q'<\infty,~ kp<n,sp'<n,s\leq k, t:= \frac{npp'}{n(p+p')-kpp'}>1$ and $\frac{1}{u}:= \min\{\frac{1}{q}+\frac{1}{q'}, 1 \}$. Further let $B^n\subset\R^n$. If $f\in W^{k,p,q}(B^n),~ g\in W^{s,p',q'}(B^n)$, then 
	\begin{align*}
	fg&\in W^{s,t,u}(B^n)
	\intertext{ and }
	||fg||_{W^{s,t,u}(B^n)}&\leq c ||f||_{W^{k,p,q}(B^n)}||g||_{W^{s,p'q'}(B^n)}
	\end{align*}
	with $c=c(B^n)$.
\end{lemma}
Furthermore, we need an optimal Sobolev embedding result.
\begin{lemma}
	Let $B^n\subset \R^n$. If $f\in W^{k,\frac{n}{k},1}(B^n)$, then $f$ is continuous on $B^n$.
\end{lemma}

Later on we also use Lorentz-Sobolev spaces $W^{k,p,q}$ with a negative exponent $k$. These are distribution spaces and for $p,q>1$ they are the dual spaces of $W^{k,p,q}$.

\begin{definition}
	Let $1<p,q<\infty,~ \frac{1}{p}+\frac{1}{p'}=\frac{1}{q}+\frac{1}{q'}=1$ and $k\in\N$. Then $W^{-k,p,q}(B^n)$ is the space of distributions $\Phi\in (C^\infty_c(B^n))'$ such that 
	\begin{align*}
	|\Phi [f]|\leq c||f||_{W^{k,p',q'}(B^n)}\qquad\forall f\in C^\infty_c(B^n).
	\end{align*}
\end{definition}

Each element of $W^{-k,p,q}$ has a representation in terms of derivatives of Lorentz functions:

\begin{lemma}\label{Lorentz Sobolev distr darstell}
	Let $1<p,q<\infty,~ k\in\N,~ B^n\subset \R^n$ and $f\in W^{-k,p,q}(B^n)$. Then there exist	$f_\alpha\in L^{p,q}(B^n)$ so that
	\begin{align*}
	f=\sum_{|\alpha|\leq k}\p^{\alpha}f_\alpha.
	\end{align*}
	Note that this representation is not unique. We define the norm on $W^{-k,p,q}(B^n)$ by 
	\begin{align*}
	||f||_{W^{-k,p,q}(B^n)} := \inf \left\{ \sum_{|\alpha|\leq k}||f_\alpha||_{L^{p,q}(B^n)} :   f= \sum_{|\alpha|\leq k} \p^{\alpha}f_\alpha  \right\}.
	\end{align*}
\end{lemma}

The definition of negative Lorentz-Sobolev spaces as dual spaces does not hold for $p, q=1$ since $L^{p,1}, L^{p',\infty}$ are not reflexive. In this case we define the space $W^{-k,p,1}$ as follows
\begin{definition}
	Let $1<p<\infty$, $k\in \N$. Then 	
	\begin{align*}
	W^{-k,p,1}(B^n):=\left \{ f= \sum_{|\alpha|\leq k} \p^{\alpha}f_\alpha : f_\alpha\in L^{p,1}(B^n)  \right\}
	\end{align*}
	with norm 
	\begin{align*}
	||f||_{W^{-k,p,1}(B^n)} := \inf \left\{ \sum_{|\alpha|\leq k}||f_\alpha||_{L^{p,1}(B^n)} :   f= \sum_{|\alpha|\leq k} \p^{\alpha}f_\alpha  \right\}
	\end{align*}
\end{definition}

Finally we have an embedding theorem and a H\"older inequality.

\begin{lemma}\label{Lorentz Sobolev embedding}
	Let $B^n\subset \R^n$, $1<p<d, ~1\leq q\leq p,~ l,s,t\in \N_0$ with $tp<n$ and $f\in W^{-s,p,q}(B^n, \wedge^l\R^n)$. Then $f\in W^{-(s+t),\frac{np}{n-tp},q}(B^n, \wedge^l\R^n)$ and 
	\begin{align*}
	||f||_{W^{-(s+t),\frac{np}{n-tp},q}(B^n)}\leq c ||f||_{W^{-s,p,q}(B^n)}.
	\end{align*}
\end{lemma}

\begin{lemma}\label{Lorentz Sobolev prod}
	Let $s,t\in\N,~t\leq s,~ 1<p,p'<\infty$ with $\frac{1}{p}+\frac{1}{p'}\leq 1$ and $tp<n,~ sp'<n,~1\leq q,q'<\infty.$ Let  $f\in W^{-t,p,q}(B^n)$ and $g\in W^{s,p',q'}(B^n)$. Then 
	\begin{align*}
	fg\in W^{-t,x,y}(B^n)
	\end{align*}
	with $x=\frac{dpp'}{n(p+p')-spp'}$ and $\frac{1}{y}=\min\{1,\frac{1}{q},\frac{1}{q'}\}$. Further 
	\begin{align*}
	||fg||_{W^{-t,x,y}(B^n)}\leq c||f||_{W^{-t,p,q}(B^n)} ||g||_{W^{s,p',q'}(B^n)}.
	\end{align*} 
\end{lemma}
More details about these spaces and proofs of the above results can be found in \cite{LongDiss}.

\subsection{A generalized Wente lemma}

A key ingredient in the proof of the main Theorem later on will be the following Wente-type lemma in the spirit of Bethuel and Ghidaglia \cite{BethuelGhidaglia93}. A fourth order version of this result can already be found in \cite{LammRiv}.

\begin{lemma}\label{Wente type lemma}
	Let $\sigma>0$, $f\in L^{\frac{2m}{2m-1-|\gamma|},1}(B^{2m}, \R^{n})$ for $|\gamma|\leq m-2$ and $P\in W^{m,2} (B^{2m}, SO(n))$ with $||d P||_{W^{m-1,2}}\leq \sigma$.
	There exists $\sigma_0>0$ such that if $\sigma<\sigma_0$ there exists a unique solution  $u\in W^{2m-1,\frac{2m}{2m-1-|\gamma|},1}(B^{2m}, M(n))$ of
	\begin{align}\label{boundary cond lemma}\begin{cases}
	\Delta (\Delta^{m-1} u\cdot P)&= \delta f \qquad\quad\text{in }B^{2m},
	\\
	\qquad\quad\Delta^{j} u&=0\qquad\quad\text{on }\p B^{2m}~~\text{ for }j=0,...,m-1,
	\end{cases}
	\end{align}
	with
	\begin{align*}
	||D^{2m-1} u||_{L^{\frac{2m}{2m-1-|\gamma|},1}(B^{2m})}+ ||u||_{L^\infty(B^{2m})} &\leq c ||f||_{L^{\frac{2m}{2m-1-|\gamma|},1}(B^{2m})}.
	\end{align*}
\end{lemma}

\begin{proof}
	
	The boundary conditions determine a solution $u$ of (\ref{boundary cond lemma}) uniquely. To see this we assume there exist solutions $u_1,~u_2$ and we let $v:= u_1-u_2$. Then $\Delta (\Delta^{m-1} v\cdot P)=0$. Testing this equation with $\Delta^{m-1} v\cdot P$ and integrating by parts gives
	\begin{align*}
	0&=\int_{B^{2m}} \Delta(\Delta^{m-1} v\cdot P) (\Delta^{m-1}v\cdot P) 
	= - \int_{B^{2m}} |D(\Delta^{m-1}v \cdot P)|^2.
	\end{align*}
	Thus we have $D(\Delta^{m-1}v \cdot P)=0$ and therefore $\Delta^{m-1}v\cdot P =const$. Because $P$ is invertible and $\Delta^{m-1}v=0$ on $\p B^{2m}$ we get $\Delta^{m-1}v=0$. Iteratively we get $v=0$ and thus $u_1=u_2$. 	
	\\
	Now we approximate $f$ by $\bar{f}\in C^{\infty}_c(\R^{2m})$ so that $\bar{f}=0$ on $\R^{2m}\setminus B^{2m}$ and  
	\begin{align*}
	||\bar{f}||_{L^{\frac{2m}{2m-1-|\gamma|},1}(\R^{2m})} 
	\leq c||f||_{L^{\frac{2m}{2m-1-|\gamma|},1}(B^{2m})}.
	\end{align*} 
	Standard $L^p$-theory and interpolation results (see \cite{Helein} Theorem 3.3.3) yield
	\begin{align*}
	|| D(\Delta^{m-1} u P)||_{L^{\frac{2m}{2m-1-|\gamma|},1}(B^{2m})}&\leq c||f||_{L^{\frac{2m}{2m-1-|\gamma|},1}(B^{2m})}.
	\end{align*}	
	With H\"older's inequality for Lorentz spaces and the embedding theorem we estimate
	\begin{align*}
	&||D \Delta^{m-1} u||_{L^{\frac{2m}{2m-1-|\gamma|},1}(B^{2m})}
	\\
	&\leq c \bigg(||f||_{L^{\frac{2m}{2m-1-|\gamma|},1}(B^{2m})} + || D^{2m-2} u||_{L^{\frac{2m}{2m-2-|\gamma|},2}(B^{2m})}||d P||_{L^{2m,2}(B^{2m})} \bigg)
	\\
	&\leq c \left(||f||_{L^{\frac{2m}{2m-1-|\gamma|},1}(B^{2m})}  + || u||_{W^{2m-1,\frac{2m}{2m-1-|\gamma|},1}(B^{2m})}||d P||_{W^{m-1,2}(B^{2m})} \right).
	\end{align*}
	We interchange derivatives and apply the Calderon-Zygmund inequality 
	\begin{align*}
	&||D^{2m-1}u||_{L^{\frac{2m}{2m-1-|\gamma|},1}(B^{2m})}
	\\
	&\leq c \left(||f||_{L^{\frac{2m}{2m-1-|\gamma|},1}(B^{2m})} + || u||_{W^{2m-1,\frac{2m}{2m-1-|\gamma|},1}(B^{2m})}||d P||_{W^{m-1,2}(B^{2m})} \right).
	\end{align*}
	Since  $||d P||_{W^{m-1,2}(B^{2m})}<\sigma$ we absorb the second term to the left-hand side. The density of $C^\infty_c(B^{2m})$  in $L^{p,q}(B^{2m})$ finishes the proof.	
	
\end{proof}

\subsection{The main result}
Before we are able to state our main result we introduce some more notation. Let $\wedge^k\R^{2m},~k\in \N_0$ be the space of $k$-forms on $\R^{2m}$. Further let 
\begin{align*}
d:W^{1,p}(\R^{2m} , \wedge^k\R^{2m} ) \rightarrow L^p(\R^{2m}, \wedge^{k+1}\R^{2m})
\end{align*}
be the exterior derivative and 
\begin{align*}
\delta : W^{1,p}(\R^{2m} , \wedge^k\R^{2m} ) \rightarrow L^p(\R^{2m},  \wedge^{k-1}\R^{2m})
\end{align*}
the codifferential. We have $dd=\delta\delta =0$ and the Laplacian is given by 
\begin{align*}
\Delta =d\delta+\delta d.
\end{align*}
If $f$ is a function, the exterior derivative of $f$ is just the gradient $\nabla f$. 
Let $0\leq k\leq 2m$ with $k\in\N$, then we let
\begin{align*}
*: \wedge^k\R^{2m}\rightarrow \wedge^{2m-k}\R^{2m}
\end{align*}
be the Hodge-Star operator. For a $k$-form $\omega$ we have 
\begin{align}\label{delta def}
\delta \omega= (-1)^{2m(k+1)+1}*d*\omega
\end{align}
and
\begin{align}\label{double star}
**:(-1)^{k(2m-k)}: \wedge^k\R^{2m}\rightarrow \wedge^k\R^{2m}.
\end{align}
(see e.g. \cite{Jost}).
\\\\
The following is the main result of this paper.		
		\begin{thm}\label{existence thm}
			Assume $m\geq 2,~n\in\N$. Let coefficient functions be given as 
			\begin{align*}
			w_k&\in W^{2k+2-m,2}(B^{2m}, \R^{n\times n})\qquad\qquad\qquad\text{for }k\in\{0,...,m-2\},
			\\
			V_k&\in W^{2k+1-m,2}(B^{2m}, \R^{n\times n}\otimes \wedge^1 \R^{2m} )\qquad\text{for }k\in \{0,...,m-1\}, \text{ where}
			\\
			V_0&= d\eta +F,~~\eta\in W^{2-m,2}(B^{2m}, so(n)),~~ F\in W^{2-m, \frac{2m}{m+1},1}(B^{2m}, \R^{n\times n}\otimes \wedge^1 \R^{2m})
			\end{align*}
			We consider the equation 
			\begin{align}\label{m poly system}
			\Delta^m u&= \sum_{k=0}^{m-1} \Delta^k \langle V_k, du\rangle + \sum_{k=0}^{m-2} \Delta^k \delta (w_k du).
			\end{align}
			For this equation, the following statements hold.
			\begin{itemize}
				\item[(i)] Let 	
				\begin{align}\label{sigma abs v w}\begin{split}
				\sigma &:= \sum_{k=0}^{m-2} ||w_k||_{W^{2k+2-m, 2}(B^{2m})} +\sum_{k=1}^{m-1}||V_k||_{W^{2k+1-m, 2}(B^{2m})} 
				\\
				&\qquad+ ||\eta||_{W^{2-m,2}(B^{2m})} + ||F||_{W^{2-m, \frac{2m}{m+1},1}(B^{2m})}.
				\end{split}
				\end{align}
				There is $\sigma_0>0$ such that whenever $\sigma<\sigma_0$, there exist
				 $\varepsilon\in W^{m,2}\cap L^\infty(B^{2m}_{1/2}; M(n))$ with
				\begin{align*}
				||\varepsilon||_{W^{m,2}(B^{2m}_{1/2})}+||\varepsilon||_{L^\infty(B^{2m}_{1/2})}\leq c\sigma,
				\end{align*}
				a function $P\in W^{m,2}(B_{1/2};SO(n))$ and a distribution $B\in W^{2-m,2}_{loc}(B^{2m}_{1/2}, \R^{n\times n}\otimes \wedge^2 \R^{2m})$ which solves
				\begin{align*}
				\delta B &= \sum_{k=0}^{m-1} \Delta^k 	((id+\varepsilon)P )V_k -\sum_{k=0}^{m-2} d \Delta^k	((id+\varepsilon)P )w_k  + d\Delta^{m-1}	((id+\varepsilon)P ).
				\end{align*}
			\item[(ii)] A function $u\in W^{m,2}(B^{2m}_{1/2}, \R^n)$ solves (\ref{m poly system}) weakly if and only if it is a distributional solution of the conservation law 
				\begin{align}\label{conservation law}\nonumber
				\delta&\bigg[\sum_{l=0}^{m-1}\Delta^l ((id+\varepsilon)P) \Delta^{m-l-1} du - \sum_{l=0}^{m-2} d\Delta^l((id+\varepsilon)P) \Delta^{m-l-1} u
				\\\nonumber
				&\qquad -\sum_{k=0}^{m-1}\sum_{l=0}^{k-1} \Delta^l ((id+\varepsilon)P )\Delta^{k-l-1} d\langle V_k, du\rangle
				\\\nonumber
				&\qquad +\sum_{k=0}^{m-1}\sum_{l=0}^{k-1} d\Delta^l ((id+\varepsilon)P)\Delta^{k-l-1} \langle V_k, du\rangle
				\\\nonumber 
				&\qquad - \sum_{k=0}^{m-2}\sum_{l=0}^{k} \Delta^l ((id+\varepsilon)P) d\Delta^{k-l-1}\delta (w_k du) 
				\\
				&\qquad+  \sum_{k=0}^{m-2}\sum_{l=0}^{k-1} d\Delta^l ((id+\varepsilon)P ) \Delta^{k-l-1}\delta (w_k du)   -\langle B,d u\rangle\bigg]=0.
				\end{align}
				
			\item[(iii)] Every weak solution $u$ of (\ref{m poly system}) is continuous. 
			\end{itemize}
		\end{thm}
A different variant of this result has been obtained earlier by Lamm and Rivi\`ere \cite{LammRiv} in the case $m=2$ and by De Longueville and Gastel \cite{LongGas} for general $m$. The key difference to these papers is that we use a small perturbation $(id+\varepsilon)P$ of the Uhlenbeck gauge matrix $P$, see Theorem \ref{Uhlenbeck thm}, to establish the conservation law. This Ansatz highlights the strong connection between the conservation law and the matrix $P$ more explicitly than the previous papers.
Another new ingredient in our approach is Lemma \ref{Wente type lemma}, a generalization of an estimate by Bethuel and Ghidaglia \cite{BethuelGhidaglia93}, which we use instead of  a Wente type result for the poly-Laplace operator. This allows for  more general elliptic operators in divergence form and  simplifies the argument.

We also remark that in a recent paper by Guo and Xiang \cite{Guo2020} it was shown that weak solutions of \eqref{m poly system} are not only continuous but even H\"older continuous for some positive exponent.

\section{Second order case}\label{section second order case}

In this section we briefly review the second order case of the main Theorem \ref{existence thm}.  We will not discuss the original proof in \cite{Riviere07} but we will focus on Rivi\`ere's  subsequent idea to establish a conservation law by using a small perturbation of the Uhlenbeck gauge matrix $P$. This proof was already sketched in \cite{RivLecture}, chapter IX, but since we will follow the same strategy in the proof of our main Theorem we decided to include this argument here.

\begin{thm}\label{riv cons neuer Ansatz thm}
	Let $n\in \N$ and $N$ be an oriented submanifold of $\R^n$.   Let $u\in W^{1,2}(B^2, N)$ be a solution of 
	\begin{align}\label{harm omega eqn thm}
	-\Delta u=\Omega\cdot \nabla u,
	\end{align}
	where $\Omega\in L^2(B^2, so(n)\otimes \wedge^1 \R^2)$
	and let $\sigma := ||\Omega||_{L^2}$. There exists $\sigma_0>0$ such that whenever $\sigma<\sigma_0$, there exist $\varepsilon\in W^{1,2}\cap L^\infty(B^2, M(n)),~P\in W^{1,2}(B^2, SO(n))$ and $\xi\in W^{1,2}(B^2,so(n))$ with 
	\begin{align*}
	||\varepsilon||_{L^\infty}+ ||\nabla \varepsilon||_{L^2} + 	||\xi||_{W^{1,2}(B^2)}+ ||\nabla P||_{L^2(B^2)}\leq c\sigma,
	\end{align*}
	and $ B\in W^{1,2}(B^2)$ that solve 
	\begin{align*}
	\nabla^\perp B= \nabla \varepsilon P -(id+\varepsilon)\nabla^\perp \xi P.
	\end{align*} 
	Further $u$ solves (\ref{harm omega eqn thm}) if and only if it is a solution of 
	\begin{align*}
	-div((id+\varepsilon)P\nabla u)=\nabla^\perp B\cdot \nabla u
	\end{align*}
	and $u$ is continuous.
\end{thm}

The proof of Theorem \ref{riv cons neuer Ansatz thm} relies heavily on Uhlenbeck's gauge theorem, see  for example \cite{Riviere07,Uhlenbeck82,Schikorra10}. 

\begin{thm}[Uhlenbeck gauge]\label{Uhlenbeck for harm}
	There exists $\sigma>0$ and $c>0$ such that for every $\Omega\in L^2 (B^2, so(n)\otimes \wedge^1 \R^2)$ satisfying $||\Omega||_{L^2(B^2)}<\sigma$ there exist $P\in W^{1,2}(B^2, SO(n))$ and $\xi\in W^{1,2}(B^2, so(n))$ such that 
	\begin{align*}
	\Omega&=P^{-1}\nabla^\perp \xi P+ P^{-1}\nabla P
	\intertext{and}
	||\xi||_{W^{1,2}(B^2)}&+ ||\nabla P||_{L^{2}(B^2)}\leq c ||\Omega||_{L^2(B^2)}.
	\end{align*}
\end{thm} 
\bigskip 

\begin{proof}[Proof of Theorem \ref{riv cons neuer Ansatz thm}:]

Assume $||\Omega||_{L^2(B^2)}<\sigma$ as in Theorem \ref{Uhlenbeck for harm}. Then we get $P\in W^{1,2}(B^2, SO(n))$, $\xi\in W^{1,2}(B^2, so(n))$ such that 
\begin{align*}
\Omega=&P^{-1}\nabla^\perp \xi P+ P^{-1}\nabla P \ \ \ \text{and}
\\
||\xi||_{W^{1,2}(B^2)}&+||\nabla P||_{L^{2}(B^2)}\leq c ||\Omega||_{L^2(B^2)}.
\end{align*} 
We multiply (\ref{harm omega eqn thm}) with $(id+\varepsilon)P$, where $\varepsilon\in W^{1,2}\cap L^\infty(B^2, M(n))$ and $id$ is the identity matrix in $\R^n$, and obtain
\begin{align}\nonumber
-(id+\varepsilon)P \Delta u&= (id+\varepsilon)P\Omega\cdot \nabla u
\\\nonumber
\Leftrightarrow -div\left[ (id+\varepsilon) P \nabla u\right]
&= \left[ -\nabla\varepsilon P + (id+\varepsilon)(-\nabla P+P\Omega)\right]\cdot \nabla u
\\\label{almost div harm}
\Leftrightarrow -div\left[ (id+\varepsilon)P  \nabla u\right]&= \left[ -\nabla \varepsilon P +(id+\varepsilon)\nabla^\perp \xi P\right]\cdot \nabla u.
\end{align}
We choose $\varepsilon\in W^{1,2}\cap L^\infty(B^2, M(n))$ such that 
\begin{align}\label{div eps harm}
div\left[ -\nabla \varepsilon P +(id+\varepsilon)\nabla^\perp \xi P\right]=0.
\end{align}
To do this we apply a fixed point argument. Let
\begin{align*}
\psi: W^{1,2}\cap L^\infty(B^2)&\rightarrow W^{1,2}\cap L^\infty(B^2)
\\
\varepsilon&\mapsto \text{solution $\lambda$ of (\ref{fixpt harm}) }
\end{align*} 
where
\begin{align}\label{fixpt harm}
\bigg\{\begin{split}
div[ \nabla\lambda P]&= \nabla( (id+\varepsilon)P)\cdot  \nabla^\perp \xi \qquad \text{in }B^2,
\\
\lambda&=0\qquad\qquad\qquad\qquad \qquad\text{on }\partial B^2.
\end{split}
\end{align}

Let $\varepsilon_1,\varepsilon_2\in W^{1,2}\cap L^\infty(B^2)$ and $\psi(\varepsilon_1)=\lambda_1,~ \psi(\varepsilon_2)=\lambda_2$ be the corresponding solutions of (\ref{fixpt harm}). Then $\Lambda:=\lambda_1-\lambda_2$ solves 
\begin{align*}
\bigg\{\begin{split}
div[ \nabla\Lambda P]&= \nabla( (\varepsilon_1-\varepsilon_2)P)\cdot  \nabla^\perp \xi \qquad \text{in }B^2,
\\
\Lambda&=0\qquad\qquad\qquad\qquad \qquad\text{on }\partial B^2.
\end{split}
\end{align*}
Since $P$ takes values in $SO(n)$ it satisfies the assumptions of Theorem 1.3 in \cite{BethuelGhidaglia93} and we have 
\begin{align*}
||\Lambda||_{L^\infty(B^2)}+||\nabla\Lambda||_{L^2(B^2)}&\leq c \bigg(||\nabla \varepsilon_1-\nabla \varepsilon_2||_{L^2(B^2)}||P||_{L^\infty(B^2)} 
\\
&\quad + ||\varepsilon_1-\varepsilon_2||_{L^\infty(B^2)} ||\nabla P||_{L^2(B^2)}\bigg)\cdot  ||\nabla\xi||_{L^2(B^2)}
\\
&\leq c\sigma \left(||\nabla \varepsilon_1-\nabla \varepsilon_2||_{L^2(B^2)} + ||\varepsilon_1-\varepsilon_2||_{L^\infty(B^2)} \right).
\end{align*}
For $\sigma$ small enough we conclude that $\psi$ is a contraction. To show that $\psi$ is a self-map from a small ball in $ W^{1,2}\cap L^\infty(B^2)$ into itself, we use again Theorem 1.3 in \cite{BethuelGhidaglia93} to get
\begin{align*}
	||\lambda||_{L^\infty(B^2)}+||\nabla\lambda||_{L^2(B^2)}&\leq c  ||\nabla\xi||_{L^2(B^2)}\big( ||\nabla \varepsilon||_{L^2(B^2)}
	\\
	&\quad +(1+||\varepsilon||_{L^{\infty}(B^2)})||\nabla P||_{L^2(B^2)}\big).
\end{align*}
The Banach fixed point theorem yields a unique $\varepsilon^*\in W^{1,2}\cap L^\infty(B^2, M(n))$ solving \eqref{fixpt harm} and hence also (\ref{div eps harm}) and with the estiamte above we get 
\begin{align*}
||\varepsilon^*||_{L^\infty}+ ||\nabla \varepsilon^*||_{L^2}\leq c\sigma.
\end{align*}
By the Poincar\'e lemma there exists $B\in W^{1,2}(B^2)$ such that 
\begin{align*}
\nabla^\perp B= -\nabla \varepsilon^* P+ (id+\varepsilon^*)\nabla^\perp \xi P
\end{align*}
and (\ref{harm omega eqn thm}) is equivalent to 
\begin{align*}
-div((id+\varepsilon^*)P \nabla u)=\nabla^\perp B \cdot \nabla u.
\end{align*}
Now that we have our equation in the desired divergence-free form, we can show the continuity of the solution $u$ using the Hodge decomposition (see Corollary 10.70 in \cite{GiaMarti})
\begin{align*}
(id+\varepsilon^*)P\nabla u= \nabla V+\nabla^\perp W
\end{align*}
and arguing as in \cite{Riviere07}.
\end{proof}

\section{Proof of Theorem \ref{existence thm}}\label{sec Proof}

We split the proof of this result into several steps and present each step in a separate subsection.

\subsection{Gauge fixing}
	
Following the work of de Longueville and Gastel in the proof of Theorem 4.1 (i) in \cite{LongGas} we repeatedly solve Neumann problems to find $\Omega\in W^{m-1,2}(B^{2m}, so(n)\otimes \wedge^1 \R^{2m})$ such that
	\begin{align}\label{Omega eta relation}
	\Delta^{m-2}\delta\Omega
	=-\eta \qquad\text{in }B^{2m} \ \ \ \text{and}
	\\
	||\Omega||_{W^{m-1,2}(B^{2m})}\leq c||\eta||_{W^{2-m,2}(B^{2m})}\leq c\sigma.
	\end{align}
	Next we need the following higher order version of the Uhlenbeck gauge fixing result which is due to De Longueville and Gastel.
	\begin{thm}[Theorem 2.4 in \cite{LongGas}]\label{Uhlenbeck thm}
			Assume that $m,n\in\N$ and $B_r\subset\R^{2m}$ is a ball of radius $r$. Then there is $\varepsilon>0$ such that for all $\Omega\in W^{m-1,2}(B_{r}, so(n)\otimes \wedge^1 \R^{2m})$ satisfying 
			\begin{align*}
			||\Omega||_{W^{m-1,2}(B_r)}<\varepsilon,
			\end{align*}
			there are functions $P\in W^{m,2}(B_{r/2};SO(n))$ and $\xi\in W^{m,2}(B_{r/2}, so(n)\otimes \wedge^2 \R^{2m})$ such that
			\begin{align}\label{Omega in Uhlenbeck thm}
			\Omega= PdP^{-1} +P \delta \xi P^{-1}
			\end{align}
			holds on $B_{r/2}$. Moreover, we have the estimate 
			\begin{align}\label{Uhlenbeck lemma abs}
			||dP||_{W^{m-1,2}(B_{r/2})} + ||\delta\xi||_{W^{m-1,2}(B_{r/2})}\leq c ||\Omega||_{W^{m-1,2}(B_r)}.
			\end{align}
		\end{thm}
We apply this result for $\sigma>0$ sufficiently small, and get $\xi\in W^{m,2}(B_{1/2}^{2m}, so(n)\otimes \wedge^2 \R^{2m})$ and $P\in W^{m,2}(B_{1/2}^{2m}, SO(n))$ such that
	\begin{align}\label{Omega decomp}\begin{split}
	dP&= P\Omega-\delta \xi P \ \ \ \text{and}
	\\
	||d P||_{W^{m-1,2}(B^{2m}_{1/2})}+ &||\delta \xi||_{W^{m-1,2}(B^{2m}_{1/2})}\leq c ||\Omega||_{L^{m-1,2}(B^{2m})}.
	\end{split}
	\end{align}
	\subsection{Rewriting the system}
	We let $\varepsilon\in W^{m,2}\cap L^\infty(B^{2m}_{1/2}, M(n))$ and we multiply (\ref{m poly system}) with $(id+\varepsilon)P$ and calculate
	\begin{align}\nonumber
	&(id+\varepsilon)P \Delta^m u= (id+\varepsilon)P \bigg[ \sum_{k=0}^{m-1} \Delta^k \langle V_k, du\rangle + \sum_{k=0}^{m-2} \Delta^k \delta (w_k du)\bigg]
	\\\nonumber
	\Leftrightarrow &\bigg[\sum_{k=0}^{m-1} \Delta^k 	((id+\varepsilon)P )V_k -\sum_{k=0}^{m-2} d \Delta^k	((id+\varepsilon)P )w_k  + d\Delta^{m-1}	((id+\varepsilon)P )\bigg]\cdot d u 
	\\\label{cons law first eqn} \begin{split}
	&= \delta\bigg[ \sum_{l=0}^{m-1}\Delta^l ((id+\varepsilon)P) \Delta^{m-l-1} du - \sum_{l=0}^{m-2} d\Delta^l((id+\varepsilon)P) \Delta^{m-l-1} u
	\\
	&\qquad -\sum_{k=0}^{m-1}\sum_{l=0}^{k-1} \Delta^l ((id+\varepsilon)P )\Delta^{k-l-1} d\langle V_k, du\rangle
	\\
	&\qquad +\sum_{k=0}^{m-1}\sum_{l=0}^{k-1} d\Delta^l ((id+\varepsilon)P)\Delta^{k-l-1} \langle V_k, du\rangle
	\\
	&\qquad - \sum_{k=0}^{m-2}\sum_{l=0}^{k} \Delta^l ((id+\varepsilon)P) d\Delta^{k-l-1}\delta (w_k du) 
	\\
	&\qquad+  \sum_{k=0}^{m-2}\sum_{l=0}^{k-1} d\Delta^l ((id+\varepsilon)P ) \Delta^{k-l-1}\delta (w_k du)  
	\bigg].
	\end{split}
	\end{align}
	The right-hand side of this system is already in divergence form, hence in order to obtain a conservation law we need to find $\varepsilon\in W^{m,2}\cap L^{\infty}(B^{2m}_{1/2}, M(n))$ such that
	\begin{align}\label{epsilon dgl}
	\delta\bigg[\sum_{k=0}^{m-1} \Delta^k 	((id+\varepsilon)P )V_k -\sum_{k=0}^{m-2} d \Delta^k	((id+\varepsilon)P )w_k  + d\Delta^{m-1}	((id+\varepsilon)P )\bigg]=0
	\end{align}
	on $B^{2m}_{1/2}$. As in section \ref{section second order case} we want to apply a fixed point argument to solve this problem. However to do this we need to have a certain control on the terms in (\ref{epsilon dgl}) and the terms involving $V_0$ are problematic. We know that $V_0=d\eta +F$ and we control $F\in W^{2-m,\frac{2m}{m+1},1}(B^{2m})$ by (\ref{sigma abs v w}) but $d\eta\in W^{1-m,2}(B^{2m})$ is a priori not bounded. Thus our goal is to remove $d\eta$.
	\\
	To do this we take a closer look at $d\Delta^{m-1}((id+\varepsilon)P)$ and note that we can rewrite the highest order term $(id+\varepsilon)d\Delta^{m-1}P$ so that it cancels $(id+\varepsilon)Pd\eta$ in (\ref{epsilon dgl}). To see this we use (\ref{delta def}), (\ref{double star}) as well as (\ref{Omega eta relation}) and (\ref{Omega decomp}) .
	\begin{align*}
	d \Delta^{m-1}P& = d\Delta^{m-2}\delta\left(P\Omega-\delta \xi P\right)
	\\
	&=d\Delta^{m-2}(dP\Omega)+ d\Delta^{m-2}(P\delta \Omega) - d\Delta^{m-2}(*d*(*d*\xi P))
	\\
	&=\sum_{i=1}^{2m-2}c_i\nabla^i P\nabla^{2m-2-i}\Omega - d(P\eta) + d\Delta^{m-2}(*(d*\xi\wedge dP))
	\\
	&= \sum_{i=1}^{2m-2}c_i\nabla^i P\nabla^{2m-2-i}\Omega - dP\eta -P(V_0-F)+ d\Delta^{m-2}(*(d*\xi\wedge dP)),
	\end{align*}
	with constants  $c_i\in \N_0$, $1\le i \le 2m-2$ and 
	\begin{align*}
	\nabla^{k}=\begin{cases}
	\Delta^{\frac{k}{2}},\qquad&\text{if $k$ even},
	\\
	d\Delta^{\frac{k-1}{2}},\qquad &\text{if $k$ odd.}
	\end{cases}
	\end{align*}
	Plugging this back into (\ref{epsilon dgl}) and rearranging we get
	\begin{align}\label{pde without V0}\nonumber
	\Delta (\Delta^{m-1}\varepsilon  \cdot P)&= \delta \bigg[-\sum_{j=1}^{2m-2}\tilde{c}_j \nabla^j \varepsilon  \nabla^{2m-1-j}P 
	-(id+\varepsilon)\bigg( \sum_{i=1}^{2m-2}c_i\nabla^i P\nabla^{2m-2-i}\Omega
	\\
	&\qquad\qquad - dP\eta +PF+ d\Delta^{m-2}(*(d*\xi\wedge dP))\bigg)
	\\\nonumber
	&\qquad -\sum_{k=1}^{m-1} \Delta^k 	((id+\varepsilon)P )V_k 
	+\sum_{k=0}^{m-2} d\Delta^k	((id+\varepsilon)P )w_k \bigg]\qquad\text{in }B^{2m}_{1/2},
	\end{align}
	where $\tilde{c}_j$ are constants in $\N_0.$ Now that we have removed the "worst" terms we want to examine this equation further and take a closer look at the function spaces of the summands. 
	We separate the $\varepsilon$ component from the rest and use the embedding results for Lorentz-Sobolev spaces  in Lemma \ref{Lorentz Sobolev embedding} and Lemma \ref{Lorentz Sobolev prod} repeatedly. We use the notation $D^k A\star D^lB$ for any linear combination of $D^k A$ and $D^l B$ and $D$ denotes the full derivative. For the first term we have
	\begin{align*}
	\sum_{j=1}^{2m-2}D^j \varepsilon \star D^{2m-1-j}P 
	&= \sum_{j=1}^{2m-2}W^{m-j,2 }\cdot W^{-m+1+j,2},
	\end{align*}
	For the third and fourth term we get 
	\begin{align*}
	(id+\varepsilon)dP\eta&= L^\infty\cdot W^{m-1,2} \cdot W^{2-m, 2}\hookrightarrow L^\infty\cdot W^{2-m, \frac{2m}{m+1},1},
	\\
	(id+\varepsilon)PF&= L^\infty\cdot L^\infty \cdot W^{2-m, \frac{2m}{m+1},1}.
	\end{align*}
	The second term is of the from 
	\begin{align*}
	& (id+\varepsilon)\bigg(\sum_{j=1}^{2m-3}D^j \Omega\star D^{2m-2-j}P + \Omega\star D^{2m-2}P\bigg)
	\\
	&= \sum_{j=1}^{2m-3}L^\infty\cdot  W^{m-1-j,2}\cdot W^{-m+2+j,2} + L^\infty\cdot W^{m-1,2}\cdot W^{2-m,2}
	\\
	&\hookrightarrow \sum_{j=1}^{m-2} L^\infty\cdot W^{-m+2+j, \frac{2m}{m+1+j},1} + \sum_{j=m-1}^{2m-3} L^\infty \cdot W^{m-1-j, \frac{2m}{3m-2-j},1}
	\\
	&\qquad+L^\infty\cdot  W^{2-m,\frac{2m}{m+1},1}
	\\
	&\hookrightarrow  L^\infty\cdot  W^{2-m,\frac{2m}{m+1},1},
	\end{align*}
	where we used Lemma \ref{Lorentz Sobolev prod} in the first step and Lemma \ref{Lorentz Sobolev embedding} with $s= m-2-j,~ p= \frac{2m}{m+1+j},~t= j$ for $j=1,...,m-2$ and $s=-m+1+j,~p=\frac{2m}{3m-2-j},~t=2m -3-j$ for $j=m-1,...,2m-3$ in the second step. 
	The fifth term follows in the same way 
	\begin{align*}
	(id+\varepsilon)d\Delta^{m-2}\left(*(dP\wedge d*\xi)\right)&= (id+\varepsilon)\sum_{j=1}^{2m-2}D^j\xi\star D^{2m-1-j}P 
	\\
	&= \sum_{j=1}^{2m-2}L^\infty \cdot W^{m-j,2}\cdot W^{-m+1+j,2}
	\\
	&\hookrightarrow L^\infty \cdot   W^{2-m,\frac{2m}{m+1},1}.
	\end{align*}
	For the last two terms we apply again Lemma \ref{Lorentz Sobolev prod} and Lemma \ref{Lorentz Sobolev embedding} with $s= m-2k-1,~p=\frac{2m}{m+2k-j},~ t=2k-j$ for $2k+1-m<m-2k+j$ and $s=2k-j-m,~p=\frac{2m}{3m-2k-1},  ~t= 2m-2k-1$ for $m-2k+j \leq 2k+1-m.$ 
	\begin{align*}
	&\sum_{k=1}^{m-1} \Delta^k 	((id+\varepsilon)P )V_k = \sum_{k=1}^{m-1}\bigg(\sum_{j=1}^{2k-1}D^j \varepsilon \star D^{2k-j}P + (id+\varepsilon)\Delta^k P + \Delta^k \varepsilon P\bigg)V_k
	\\
	&= \sum_{k=1}^{m-1}\sum_{j=1}^{2k-1}W^{m-j,2}\cdot W^{m-2k+j,2} \cdot W^{2k+1-m,2}+\sum_{k=1}^{m-1} L^\infty\cdot W^{m-2k,2} \cdot W^{2k+1-m,2}
	\\
	&\hookrightarrow \sum_{\substack{j,k\in\N,  ~j\leq 2k-1 , ~k\leq m-1\\2k+1-m<m-2k+j}} W^{m-j,2}\cdot W^{2k+1-m,\frac{2m}{m+2k-j}} \\
	&\quad +\sum_{\substack{j,k\in\N,~ j\leq 2k-1,~k\leq m-1\\m-2k+j\leq 2k+1-m}} W^{m-j,2} \cdot W^{m-2k+j, \frac{2m}{3m-2k-1}}
	\\
	&\quad + \sum_{\substack{k\in\N,~k\leq m-1 \\ 2k+1-m<m-2k}}L^\infty\cdot W^{2k+1-m,\frac{2m}{m+2k}, 1} + \sum_{\substack{k\in\N,~k\leq m-1 \\ m-2k\leq 2k+1-m}}  L^\infty\cdot  W^{m-2k,\frac{2m}{3m-2k-1},1}
	\\
	&\hookrightarrow \sum_{j=1}^{2m-3}W^{m-j,2} \cdot W^{-m+1+j,2} + L^\infty\cdot W^{2-m, \frac{2m}{m+1},1}
	\end{align*}
	and analogously
	\begin{align*}
	&\sum_{k=0}^{m-2} \nabla \Delta^k	((id+\varepsilon)P )w_k 
	\\
	&= \sum_{k=0}^{m-2}\bigg( \sum_{j=1}^{2k} D^j \varepsilon\star D^{2k+1-j}P  +(id+\varepsilon)\delta \Delta^k P + \delta\Delta^k \varepsilon P\bigg) w_k
	\\
	&= \sum_{k=0}^{m-2} \sum_{j=1}^{2k} W^{m-j,2}\cdot W^{m-2k-1+j,2}\cdot W^{2k+2-m,2} + \sum_{k=0}^{m-2}L^\infty \cdot W^{m-2k+1}\cdot W^{2k+2-m,2}
	\\
	&\hookrightarrow \sum_{j=1}^{2m-3}W^{m-j,2} \cdot W^{-m+1+j,2} + L^\infty\cdot W^{2-m, \frac{2m}{m+1},1}.
	\end{align*}
	Observe that all terms on the right-hand side of (\ref{pde without V0}) consist of products $W^{m-j,2}\cdot W^{j+1-m,2}, ~j=1,...,2m-2$ and $L^\infty\cdot W^{2-m, \frac{2m}{m+1},1}$. Thus we can simplify (\ref{pde without V0}) further and write
	\begin{align}\label{pde mit Kj}
	\Delta(\Delta^{m-1} \varepsilon\cdot  P)&=\delta \bigg( \sum_{j=1}^{2m-2} D^j\varepsilon\star K_j + (id+\varepsilon)\star K_0 \bigg) 
	\end{align}
	with $K_j\in W^{j+1-m,2}(B^{2m}_{1/2}), ~K_0\in W^{2-m, \frac{2m}{m+1},1}(B^{2m}_{1/2})$. 
	Moreover with  (\ref{Uhlenbeck lemma abs}) and (\ref{sigma abs v w}) we estimate 
	\begin{align}\label{abs Ko Kj}
	||K_0||_{W^{2-m,\frac{2m}{m+1},1}(B^{2m}_{1/2})}+\sum_{j=1}^{2m-2}||K_j||_{W^{j+1-m,2}(B^{2m}_{1/2})}\leq c\sigma.
	\end{align}
	However the equation still contains distributions. To take care of these we apply the same technique as de Longueville and Gastel and use the representation of negative Lorentz-Sobolev spaces (see Lemma \ref{Lorentz Sobolev distr darstell}).
	\begin{align}\label{rep epsilon}\begin{split}
	\varepsilon&= \sum_{|\alpha|\leq m-2}\p^{\alpha} \varepsilon_\alpha, \qquad\qquad \varepsilon_\alpha\in W^{2m-1,\frac{2m}{2m-1-|\alpha|},1}(B^{2m}_{1/2}),
	\\
	K_0&= \sum_{|\alpha|\leq m-2}\p^{\alpha} K^\alpha_0,\qquad\quad K^\alpha_0\in L^{\frac{2m}{m+1},1}(B^{2m}_{1/2}),
	\\
	K_j&= \sum_{|\alpha|\leq m-1-j}\p^{\alpha} K^\alpha_j,\qquad K^\alpha_j\in L^2(B^{2m}_{1/2}).
	\end{split}
	\end{align}
	Together with (\ref{abs Ko Kj}) we get
	\begin{align}\label{abs Kj alpha K0 beta dis}\begin{split}
	\sum_{|\alpha|\leq m-1-j}||K_j^\alpha||_{L^2(B_{1/2}^{2m})} &\leq c ||K_j||_{W^{j+1-m,2}(B_{1/2}^{2m})}\leq c\sigma,
	\\
	\sum_{|\alpha|\leq m-2}||K_0^\alpha||_{L^{\frac{2m}{m+1},1}(B_{1/2}^{2m})} &\leq c ||K_0||_{W^{2-m,\frac{2m}{m+1},1}(B_{1/2}^{2m})}\leq c\sigma.
	\end{split}
	\end{align}
	Note that we assume  $\varepsilon\in W^{m+1,\frac{2m}{m+1},1}$ for this representation, which is slightly better than the original assumption $\varepsilon\in W^{m,2}\cap L^\infty$. We will see that we can solve (\ref{pde without V0}) in this better space and since $W^{m+1, \frac{2m}{m+1}, 1}(B^{2m})\hookrightarrow W^{m,2}\cap L^\infty(B^{2m})$ we get the desired result.
	\\\\
	This new representation allows us to shift derivatives away from the distributional part. Let $c_{\alpha\gamma}, ~c_{\beta\gamma}\in \Z$. With the product rule we get for $j=1,...,m-2$ 
	\begin{align*}
	D^j\varepsilon \star K_j=\sum_{\substack{|\alpha|\leq m-2\\ |\beta|\leq m-1-j}} D^j\p^\alpha\varepsilon_\alpha\star\p^\beta K_j^\beta = \sum_{\substack{|\alpha|\leq m-2\\ |\beta|\leq m-1-j}}\sum_{\gamma\leq \beta}\p^\gamma \big(c_{\beta\gamma}\p^{\beta-\gamma}\p^\alpha D^j \varepsilon_\alpha\star K^\beta_j \big)
	\end{align*}
	The case $j=0$ follows analogously 
	\begin{align*}
	(id +\varepsilon)\star K_0&=\sum_{|\gamma|\leq m-2}\p^\gamma K^{\gamma}_0 + \sum_{\substack{|\alpha|\leq m-2\\ |\beta|\leq m-2}}\p^\alpha\varepsilon_\alpha\star \p^\beta K^{\beta}_0
	\\
	&= \sum_{|\gamma|\leq m-2}\p^\gamma K^{\gamma}_0 +\sum_{\substack{|\alpha|\leq m-2\\ |\beta|\leq m-2}} \sum_{\gamma\leq \beta}\p^\gamma \big(c_{\beta\gamma}\p^{\beta-\gamma}\p^\alpha\varepsilon_\alpha\star  K^{\beta}_0\big).
	\end{align*}
	For $j=m-1,...,2m-2$ with $|\alpha|\leq j+1-m$ we get
	\begin{align*}
	D^j\varepsilon\star K_j=\sum_{\substack{|\alpha|\leq m-2}} D^j\p^\alpha\varepsilon_\alpha\star K_j = \sum_{\substack{|\alpha|\leq m-2}}\sum_{\gamma\leq \alpha}\p^\gamma \big(c_{\alpha\gamma}D^j \varepsilon_\alpha\star\p^{\alpha-\gamma} K_j \big).
	\end{align*}	
	If $|\alpha|>j+1-m$ we choose $\beta\leq \alpha$ with $|\beta|=j+1-m$ and 
	\begin{align*}
	D^j\varepsilon \star K_j&=\sum_{\substack{|\alpha|\leq m-2}} D^j\p^\alpha\varepsilon_\alpha\star K_j
	\\
	&= \sum_{\substack{|\alpha|\leq m-2}}\sum_{\substack{\gamma\leq \beta\\|\beta|=j+1-m}}\p^\gamma \big(c_{\beta\gamma} \p^{\alpha-\beta}D^j\varepsilon_\alpha\star\p^{\beta-\gamma} K_j \big).
	\end{align*}
	We rewrite the left-hand side of (\ref{pde mit Kj}) in the same way.
	\begin{align*}
	\Delta (\Delta^{m-1}&\varepsilon\cdot  P)= \sum_{|\alpha|\leq m-2} \Delta(\Delta^{m-1}\p^\alpha\varepsilon_\alpha \cdot P)
	\\
	&=\sum_{|\alpha|\leq m-2}\sum_{\gamma\leq \alpha}\p^\gamma \Delta(c_{\alpha\gamma}\Delta^{m-1}\varepsilon_\alpha\p^{\alpha-\gamma} P)
	\\
	&=\sum_{|\gamma|\leq m-2}\p^{\gamma }\Delta(\Delta^{m-1}\varepsilon_\gamma\cdot P) +\sum_{|\alpha|\leq m-2}\sum_{\gamma< \alpha}\p^\gamma \Delta(c_{\alpha\gamma}\Delta^{m-1}\varepsilon_\alpha\p^{\alpha-\gamma} P).
	\end{align*}
	For the last term note that $P\in W^{m,2}(B^{2m}_{1/2}, SO(n))$. Thus we identify $P$ with $K_{2m-1}$ and write
	\begin{align*}
	&\sum_{|\alpha|\leq m-2}\sum_{\gamma< \alpha}\p^\gamma \Delta(c_{\alpha\gamma}\Delta^{m-1}\varepsilon_\alpha\p^{\alpha-\gamma} P)
	\\
	&= \delta\bigg[ \sum_{|\alpha|\leq m-2}\sum_{\gamma< \alpha}\sum_{i=0}^{1}\p^\gamma \left( c_{\alpha\gamma} D^{2m-2-i}\varepsilon_\alpha\star\p^{\alpha-\gamma}D^{1-i}K_{2m-1} \right)\bigg].
	\end{align*}
	Putting all of this together we get an equation equivalent to (\ref{epsilon dgl})
	\begin{align*}
		\sum_{|\gamma|\leq m-2}\p^{\gamma }&\Delta(\Delta^{m-1}\varepsilon_\gamma \cdot P)
	\\
	&=  \delta\bigg[\sum_{|\gamma|\leq m-2}\p^\gamma K^{\gamma}_0 +\sum_{\substack{|\alpha|\leq m-2\\ |\beta|\leq m-2}} \sum_{\gamma\leq \beta}\p^\gamma \big(c_{\beta\gamma}\p^{\beta-\gamma}\p^\alpha\varepsilon_\alpha\star  K^{\beta}_0\big)
	\\
	&\qquad + \sum_{j=1}^{m-2} \sum_{\substack{|\alpha|\leq m-2\\ |\beta|\leq m-1-j}}\sum_{\gamma\leq \beta}\p^\gamma \big(c_{\beta\gamma}\p^{\beta-\gamma}\p^\alpha D^j \varepsilon_\alpha\star K^\beta_j \big)
	\\
	&\qquad +\sum_{\substack{j=m-1 \\ |\alpha|\leq j+1-m}}^{2m-2}\sum_{\substack{|\alpha|\leq m-2}}\sum_{\gamma\leq \alpha}\p^\gamma \big(c_{\alpha\gamma}D^j \varepsilon_\alpha\star\p^{\alpha-\gamma} K_j \big)
	\\
	&\qquad +\sum_{\substack{j=m-1 \\ |\alpha|>j+1-m}}^{2m-2}\sum_{\substack{|\alpha|\leq m-2}}\sum_{\substack{\gamma\leq \beta\\|\beta|= j+1-m}}\p^\gamma \big(c_{\beta\gamma} \p^{\alpha-\beta}D^j\varepsilon_\alpha\star\p^{\beta-\gamma} K_j \big)
	\\
	&\qquad+
	\sum_{i=0}^{1} \sum_{|\alpha|\leq m-2}\sum_{\gamma< \alpha}\p^\gamma \left(c_{\alpha\gamma} D^{2m-2-i}\varepsilon_\alpha\star\p^{\alpha-\gamma}D^{1-i}K_{2m-1} \right)
	\bigg].
	\end{align*}
	We simplify this further by setting
	\begin{align}\label{pde e gamma}
	\sum_{|\gamma|\leq m-2}\p^{\gamma }\Delta(\Delta^{m-1}\varepsilon_\gamma \cdot P)
	=: \delta\bigg[ \sum_{|\gamma|\leq m-2}\p^{\gamma}\left( \langle\varepsilon, K\rangle_\gamma +K_0^\gamma\right) \bigg]
	\end{align}
	with 
	\begin{align}\label{K0 e,K abs}\begin{split}
	||K_0^\gamma||_{L^{\frac{2m}{2m-1-|\gamma|},1}(B^{2m}_{1/2})}&+||\langle \varepsilon, K\rangle_\gamma||_{L^{\frac{2m}{2m-1-|\gamma|},1}(B^{2m}_{1/2})}
	\\
	&\leq c \sigma \left(\sum_{|\alpha|\leq m-2}||\varepsilon_\alpha||_{W^{2m-1, \frac{2m}{2m-1-|\alpha|},1}(B^{2m}_{1/2})} +1 \right)
	\end{split}
	\end{align}
	for every $\gamma$ with $|\gamma|\leq m-2$. 
	To see this last inequality we use (\ref{abs Kj alpha K0 beta dis}) and estimate each term separately
	\begin{align*}
	||K_0^\gamma||_{L^{\frac{2m}{2m-1-|\gamma|},1}(B^{2m}_{1/2})}
	\leq c ||K_0^\gamma||_{L^{\frac{2m}{m+1},1}(B^{2m}_{1/2})}\leq c\sigma;
	\end{align*}
	since $K_0^\gamma \in L^{\frac{2m}{m+1},1}$ and $ L^{\frac{2m}{m+1},1}\hookrightarrow L^{\frac{2m}{2m-1-|\gamma|},1}(B^{2m}_{1/2}) $ by Lemma \ref{embedding lorentz exponent}. Further we have
	\begin{align*}
	W^{2m-1-|\beta|+|\gamma|-|\alpha|-j,\frac{2m}{2m-1-|\alpha|},1}\hookrightarrow L^{\frac{2m}{j+|\beta|-|\gamma|},1} \hookrightarrow L^{\frac{2m}{m-1-|\gamma|},1}(B^{2m}_{1/2})
	\end{align*}
	by Lemma \ref{pos. LorentzSobolev Embedding} and Lemma \ref{embedding lorentz exponent} since  $|\beta|\leq m-j-1$. With Lemma \ref{Holder for Lorentz} and \ref{embedding lorentz exponent} we have  $L^{\frac{2m}{m-1-|\gamma|},1}\cdot L^2\hookrightarrow L^{\frac{2m}{2m-1-|\gamma|},1}$ and since $\gamma \leq \beta$
	\begin{align*}
	&||\p^{\beta-\gamma}\p^\alpha D^j \varepsilon_\alpha\star K^\beta_j||_{L^{\frac{2m}{2m-1-|\gamma|},1}(B^{2m}_{1/2})}
	\\
	& \leq c ||\varepsilon_\alpha||_{W^{2m-1-j-|\alpha|-|\beta|+|\gamma|, 	\frac{2m}{2m-1-|\alpha|},1}(B^{2m}_{1/2})}||K_j^\beta||_{L^2(B^{2m}_{1/2})}
	\\
	&\leq c\sigma ||\varepsilon_\alpha||_{W^{2m-1,\frac{2m}{2m-1-|\alpha|},1}(B^{2m}_{1/2})}.
	\end{align*}
	The remaining terms follow in a similar way. With Lemma \ref{pos. LorentzSobolev Embedding}
	\begin{align*}
	W^{2m-1-|\beta|+|\gamma|-|\alpha|, \frac{2m}{2m-1-|\alpha|},1}\hookrightarrow L^{\frac{2m}{|\beta|-|\gamma|},1}(B^{2m}_{1/2})
	\end{align*}
	and by Lemma \ref{Holder for Lorentz} and Lemma \ref{embedding lorentz exponent} with $|\beta|\leq m-2$
	\begin{align*}
	L^{\frac{2m}{|\beta|-|\gamma|},1}\cdot L^{\frac{2m}{m+1},1}  \hookrightarrow L^{\frac{2m}{m+|\beta|-|\gamma|+1},1}\hookrightarrow L^{\frac{2m}{2m-1-|\gamma|},1} (B^{2m}_{1/2}).
	\end{align*}
	With this and $\gamma \leq \beta $
	\begin{align*}
	&|| \p^{\beta-\gamma}\p^\alpha\varepsilon_\alpha\star K^{\beta}_0||_{L^{\frac{2m}{2m-1-|\gamma|},1}}
	\\
	& \leq 
	c ||\varepsilon_\alpha||_{W^{2m-1-|\alpha|-|\beta|+|\gamma|, \frac{2m}{2m-1-|\alpha|},1}}||K_0^\beta||_{L^2}
	\leq c\sigma  ||\varepsilon_\alpha||_{W^{2m-1,\frac{2m}{2m-1-|\alpha|},1}}.
	\end{align*}
	For the next term we have with Lemma \ref{pos. LorentzSobolev Embedding} and Lemma \ref{Holder for Lorentz}
	\begin{align*}
	W^{2m-1-j, \frac{2m}{2m-1-|\alpha|},1 }\cdot W^{j+1-m-|\alpha|+|\gamma|,2}
	&\hookrightarrow L^{\frac{2m}{j-|\alpha|},1}\cdot L^{\frac{2m}{2m+|\alpha|-|\gamma|-j-1},2}
	\\
	&\hookrightarrow L^{\frac{2m}{2m-1-|\gamma|},1}(B^{2m}_{1/2})
	\end{align*}
	so that with $\gamma \leq \alpha$
	\begin{align*}
	&||D^j \varepsilon_\alpha\star\p^{\alpha-\gamma} K_j ||_{L^{\frac{2m}{2m-1-|\gamma|},1}}
	\\
	&\leq c ||\varepsilon_\alpha||_{W^{2m-1-j, \frac{2m}{2m-1-|\alpha|},1}}||K_j||_{W^{j+1-m+|\gamma|-|\alpha|,2}}
	\leq c\sigma  ||\varepsilon_\alpha||_{W^{2m-1,\frac{2m}{2m-1-|\alpha|},1}}.
	\end{align*}
	In the fifth term we use $|\beta|= j+1-m$, Lemma \ref{pos. LorentzSobolev Embedding} and Lemma \ref{Holder for Lorentz} to get
	\begin{align*}
	W^{2m-1-|\alpha|+|\beta|-j, \frac{2m}{2m-1-|\alpha|},1}\cdot W^{j+1-m-|\beta|+|\gamma|,2}&\hookrightarrow L^{\frac{2m}{m-1},1} \cdot L^{\frac{2m}{m-|\gamma|},2}
	 \\
	 &\hookrightarrow L^{\frac{2m}{2m-1-|\gamma|},1}(B^{2m}_{1/2})
	\end{align*}
	and
	\begin{align*}
	&||\p^{\alpha-\beta}D^j\varepsilon_\alpha\star\p^{\beta-\gamma} K_j||_{L^{\frac{2m}{2m-1-|\gamma|},1}}
	\\
	&\leq c ||\varepsilon_\alpha||_{W^{2m-1-j-|\alpha|+|\beta|, \frac{2m}{2m-1-|\alpha|},1}}||K_j||_{W^{j+1-m+|\gamma|-|\beta|,2}}
	\\
	&\leq c\sigma  ||\varepsilon_\alpha||_{W^{2m-1,\frac{2m}{2m-1-|\alpha|},1}}.
	\end{align*}
	Finally we estimate for $i=0,1$ with (\ref{Uhlenbeck lemma abs}) and $\gamma \leq \alpha$
	\begin{align*}
	&||D^{2m-2-i}\varepsilon_\alpha\star\p^{\alpha-\gamma}D^{1-i}K_{2m-1}||_{L^{\frac{2m}{2m-1-|\gamma|},1}}
	\\
	&\qquad\leq ||\varepsilon_\alpha||_{W^{1+i,\frac{2m}{2m-1-|\alpha|},1}} ||P||_{W^{m-|\alpha|+|\gamma|-1+i,2}}
	\leq c\sigma  ||\varepsilon_\alpha||_{W^{2m-1,\frac{2m}{2m-1-|\alpha|},1}}
	\end{align*}	
	and this proves (\ref{K0 e,K abs}). 
\subsection{The fixed point argument}
	Instead of solving (\ref{pde e gamma}) we solve the system
	\begin{align}\label{pde fixed point eps}
	\Delta(\Delta^{m-1}\varepsilon_\gamma\cdot P)=  \delta\left( \langle\varepsilon, K\rangle_\gamma +K_0^\gamma\right)\qquad\text{for every $\gamma$ with $|\gamma|\leq m-2$.}
	\end{align}
	To do this we apply a fixed point argument: Let  $X_\gamma:=\{ u\in M(n): || u||_{W^{2m-1,\frac{2m}{2m-1-|\gamma|},1}}(B^{2m}_{1/2})<\infty \} $ and $X=\oplus_{|\gamma|\leq m-2}X_\gamma$. We define maps $\psi_\gamma : X_\gamma\rightarrow X_\gamma$  by 
	\begin{align*}
	\psi_\gamma :\varepsilon_\gamma\mapsto \text{ solution $\lambda_\gamma$ of (\ref{fix pt equ})}
	\end{align*}
	with
	\begin{align}\label{fix pt equ}
	\begin{cases}
	\Delta(\Delta^{m-1}\lambda_\gamma \cdot P)&= \delta \left( \langle\varepsilon, K\rangle_\gamma +K_0^\gamma\right)\qquad\text{in }B_{1/2}^{2m},
	\\
	\qquad\qquad\Delta^{j} \lambda_\gamma&=0\qquad\qquad\qquad\qquad\quad\text{on }\p B_{1/2}^{2m}\text{ for }j=0,...,m-1.
	\end{cases}
	\end{align} 
	Let $\hat{\lambda}=\sum_{|\gamma|\leq m-2} \lambda_\gamma$ and $\hat{\varepsilon}=\sum_{|\gamma|\leq m-2}\varepsilon_\gamma$, where $\lambda_\gamma$ is a solution of (\ref{fix pt equ}) for every $\gamma$ with corresponding $\varepsilon_\gamma$. Let $\Psi= \oplus_{|\gamma|\leq m-2}\psi_\gamma$ and
	\begin{align*}
	\mu:= ||\hat{\varepsilon}||_X:= \sum_{|\gamma|\leq m-2}||D^{2m-1} \varepsilon_\gamma||_{L^{\frac{2m}{2m-1-|\gamma|},1}(B_{1/2}^{2m})}.
	\end{align*}
	We apply Lemma \ref{Wente type lemma} and (\ref{K0 e,K abs}) to estimate
	\begin{align*}
	||D^{2m-1}\lambda_\gamma||_{L^{\frac{2m}{2m-1-|\gamma|},1}(B_{1/2}^{2m})} 
	&\leq c ||\langle \varepsilon, K\rangle_\gamma + K_0^\gamma||_{L^{\frac{2m}{2m-1-|\gamma|},1}(B_{1/2}^{2m})}
	\\
	&\leq c\sigma \left(\sum_{|\gamma|\leq m-2}||\varepsilon_\gamma||_{W^{2m-1, \frac{2m}{2m-1-|\gamma|},1}(B_{1/2}^{2m})} +1 \right)
	\\
	&\leq c_1 \sigma(\mu+1).
	\end{align*}
	We choose $\sigma < \frac{\mu}{2c_1(\mu+1)}$ to get
	\begin{align*}
	||\hat{\lambda}||_X\leq \frac{\mu}{2}.
	\end{align*}
	Next we show that $\psi_\gamma$ is a contraction. Let $\lambda_\gamma^1,~\lambda_\gamma^2$ be solutions of (\ref{fix pt equ}) with $\varepsilon_\gamma^1,~\varepsilon_\gamma^2$ respectively. Then $\Lambda_\gamma:= \lambda_\gamma^1-\lambda_\gamma^2$ is a solution of 
	\begin{align*}
	\begin{cases}
	\Delta(\Delta^{m-1}\Lambda_\gamma\cdot P)&= \delta \left( \langle\varepsilon^1 -\varepsilon^2, K\rangle_\gamma\right)\qquad\text{in }B_{1/2}^{2m},
	\\
	\qquad\qquad\Delta^{j} \Lambda_\gamma&=0\qquad\qquad\qquad\qquad\quad\text{on }\p B_{1/2}^{2m}\text{ for }j=0,...,m-1.
	\end{cases}
	\end{align*}	
	Applying Lemma \ref{Wente type lemma} and (\ref{K0 e,K abs}) again yields
	\begin{align*}
	||D^{2m-1}\lambda_\gamma^1 &-D^{2m-1}\lambda_\gamma^2||_{L^{\frac{2m}{2m-1-|\gamma|},1 }(B_{1/2}^{2m})}
	\\
	&\leq c\sigma\sum_{|\gamma|\leq m-2} ||\varepsilon_\gamma^1-\varepsilon_\gamma^2||_{W^{2m-1,\frac{2m}{2m-1-|\gamma|},1}(B_{1/2}^{2m})}.
	\end{align*}
	With this we have 
	\begin{align*}
	||\hat{\lambda}^1-\hat{\lambda}^2||_X\leq c_2\sigma ||\hat{\varepsilon}^1-\hat{\varepsilon}^2||_X.
	\end{align*}
	Choosing $\sigma<\min\{\frac{\mu}{2c_1(\mu+1)},\frac{1}{2c_2}\}$ shows that $\Psi$ is a contraction.
	Now we can apply the Banach fixed point theorem which yields a unique $\hat{\varepsilon}^*\in X$ solving (\ref{pde fixed point eps}) and by Lemma \ref{Wente type lemma} and (\ref{K0 e,K abs}) 
	\begin{align*}
	\sum_{|\gamma|\leq m-2} ||\varepsilon^*_\gamma||_{W^{2m-1,\frac{2m}{2m-1-|\gamma|},1}(B^{2m}_{1/2})}\leq c\sigma.
	\end{align*}
	Thus we have 
	\begin{align}\label{0=div(...) abkuerzung}
	0&=\delta \left( d\Delta^{m-1}\varepsilon^*_\gamma \cdot P - \langle \varepsilon^*, K\rangle_\gamma + K^\gamma_0\right)
	\end{align}
	for every $\gamma$ with $|\gamma|\leq m-2$. What is left to show is that these $\varepsilon^*_\gamma$ are the Sobolev functions in the representation  (\ref{rep epsilon}) of $\varepsilon$ and this $\varepsilon$ solves (\ref{epsilon dgl}).  
\subsection{Going back to the original system}
	In order to go back to our original system, we reverse the abbreviations we made at the beginning to get a detailed look at (\ref{0=div(...) abkuerzung}). To do this we go back to (\ref{pde without V0}). As we have seen before, each term of this equation is a product of a distribution and a Sobolev function. More precisely, the terms are of the form $L^\infty\cdot W^{2-m,\frac{2m}{m-1}}$ and $W^{m-k,2}\cdot W^{-m+1+k,2},~ k=1,...,2m-2$.  We use the following representations for the distributions according to Lemma \ref{Lorentz Sobolev distr darstell}
	\begin{align*}
	FP -d\Delta^{m-2}\delta &(\Omega P) + d\Delta^{m-2}\delta \Omega P - d \Delta^{m-2}(* (dP\wedge d* \xi))
	\\
	= \sum_{|\alpha|\leq m-2}  &\bigg( FP -d\Delta^{m-2}\delta (\Omega P) + d\Delta^{m-2}\delta \Omega P - d \Delta^{m-2}(* (dP\wedge d* \xi))\bigg)^\alpha, 
	\\
	\bigg( FP -d\Delta^{m-2}&\delta (\Omega P) + d\Delta^{m-2}\delta \Omega P - d \Delta^{m-2}(* (dP\wedge d* \xi))\bigg)^\alpha\in L^{\frac{2m}{m+1},1}(B^{2m}_{1/2})
	\\[8pt]
	\Delta^k P\cdot V_k&=\sum_{|\alpha|\leq m-2} \p^\alpha (\Delta^k P V_k)^\alpha,\qquad(\Delta^k PV_k)^\alpha\in L^{\frac{2m}{m+1},1}(B^{2m}_{1/2}),~k\neq 0
	\\
	d\Delta^kPw_k &=\sum_{|\alpha|\leq m-2} \p^\alpha (d\Delta^k P w_k)^\alpha,\qquad(d\Delta^k Pw_k)^\alpha\in L^{\frac{2m}{m+1},1}(B^{2m}_{1/2})
	\\
	\nabla^{2k-l}P\cdot V_k&=\sum_{|\alpha|\leq m-1-l} \p^\alpha (\nabla^{2k-l}PV_k)^\alpha, \quad(\nabla^{2k-l}PV_k)^\alpha \in L^{2}(B^{2m}_{1/2}),~k\neq 0  
	\\
	\nabla^{2k+1-l} P\cdot w_k&= \sum_{|\alpha|\leq m-1-l} \p^\alpha (\nabla^{2k+1-l}Pw_k)^\alpha ,\quad (\nabla^{2k+1-l}Pw_k)^\alpha \in L^{2}(B^{2m}_{1/2}),
	\\
	\nabla^{2m-1-k}P&= \sum_{|\alpha|\leq m-1-k} \p^\alpha (\nabla^{2m-1-k}P)^\alpha, \quad\qquad ~(\nabla^{2m-1-k}P)^\alpha\in L^2(B^{2m}_{1/2}).
	\end{align*}
	Then we shift derivatives to get an equation of the form 
	$\sum_{|\gamma|\leq m-2}\p^\gamma(...)_\gamma=0$
	as in (\ref{pde e gamma}). 
	Using this we see that  (\ref{0=div(...) abkuerzung}) is equivalent to 
	\begin{align*}
	&0=\delta \bigg[ \sum_{\substack{1\leq k\leq m-2\\|\alpha|\leq m-1-k}}c_{k,\alpha\gamma} \p^{\alpha-\gamma}\nabla^k\p^\beta\varepsilon^*_\beta(\nabla^{2m-1-k}P)^\alpha
	\\
	&\qquad + \sum_{\substack{m-1\leq k\leq 2m-1\\|\alpha|\leq k+1-m}}c_{k,\alpha\gamma} \nabla^k\varepsilon^*_\alpha \p^{\alpha-\gamma}\nabla^{2m-1-k}P
	\\
	&\qquad +\sum_{\substack{m-1\leq k\leq 2m-1\\|\alpha|>m-1-k\\|\beta|=m-1-k}}c_{k,\beta\gamma} \p^{\alpha-\beta}\nabla^k\varepsilon^*_\alpha \p^{\beta-\gamma}\nabla^{2m-1-k}P
	\\
	&\qquad + \bigg( FP -d\Delta^{m-2}\delta (\Omega P) + d\Delta^{m-2}\delta \Omega P - d \Delta^{m-2}(* (dP\wedge d* \xi))\bigg)^\gamma 
	\\
	&\qquad + \sum_{|\alpha|,|\beta|\leq m-2} c_{\beta \gamma} \p^{\beta-\gamma}\p^\alpha \varepsilon^*_\alpha \bigg( FP -d\Delta^{m-2}\delta (\Omega P) + d\Delta^{m-2}\delta \Omega P 
	\\
	&\qquad\qquad\qquad- d \Delta^{m-2}(* (dP\wedge d* \xi))\bigg)^\beta
	\\
	&\qquad+\sum_{k=1}^{m-1}( \Delta ^kP V_k)^\gamma 
	+ \sum_{k=0}^{m-1} \sum_{\substack{|\alpha|,|\beta|\leq m-2}} c_{\beta\gamma} \p^{\beta-\gamma}\p^\alpha\varepsilon^*_\alpha (\Delta^k PV_k)^\beta
	\\
	&\qquad +\sum_{k=1}^{m-1} \sum_{\substack{1\leq l\leq m-2\\l\leq 2k\\|\alpha|\leq l+1-m}} c_{l, \alpha\gamma} \p^{\alpha-\gamma} \nabla^l\p^\beta \varepsilon^*_\beta (\nabla^{2k-l}P V_k)^\alpha
	\\
	&\qquad +\sum_{k=1}^{m-1} \sum_{\substack{m-1\leq l\leq 2m-2\\ l\leq 2k\\ |\alpha|\leq l+1-m}} c_{l,\alpha\gamma} \nabla^{l}\varepsilon^*_\alpha \p^{\alpha-\gamma}\nabla^{2k-l}P V_k
	\\
	&\qquad +\sum_{k=1}^{m-1}  \sum_{\substack{m-1\leq l\leq 2m-2\\ l\leq 2k\\|\alpha|>l+1-m\\|\beta|=l+1-m}} c_{l,\beta\gamma} \p^{\alpha-\beta}\nabla^l \varepsilon^*_\alpha \p^{\beta-\gamma}\nabla^{2k-l}P V_k
	\\
	&\qquad -\sum_{k=0}^{m-2}(d\Delta ^kP w_k)^\gamma 
	-\sum_{k=0}^{m-2} \sum_{\substack{|\alpha|,|\beta|\leq m-2}} c_{\beta\gamma} \p^{\beta-\gamma}\p^\alpha\varepsilon^*_\alpha (d\Delta^k Pw_k)^\beta
	\\
	&\qquad - \sum_{k=0}^{m-2} \sum_{\substack{1\leq l\leq m-2\\l\leq 2k+1\\|\alpha|\leq l+1-m}} c_{l, \alpha\gamma} \p^{\alpha-\gamma} \nabla^l\p^\beta \varepsilon^*_\beta (\nabla^{2k+1-l}P w_k)^\alpha
	\\
	&\qquad -\sum_{k=0}^{m-2} \sum_{\substack{m-1\leq l\leq 2m-3\\ l\leq 2k+1\\ |\alpha|\leq l+1-m}} c_{l,\alpha\gamma} \nabla^{l}\varepsilon^*_\alpha \p^{\alpha-\gamma}\nabla^{2k+1-l}P w_k
	\\
	&\qquad -\sum_{k=0}^{m-2}  \sum_{\substack{m-1\leq l\leq 2m-3\\ l\leq 2k+1\\|\alpha|>l+1-m\\|\beta|=l+1-m}} c_{l,\beta\gamma} \p^{\alpha-\beta}\nabla^l \varepsilon^*_\alpha \p^{\beta-\gamma} \nabla^{2k+1-l}P w_k
	\bigg]
	\\
	&=:\delta [...]_\gamma.
	\end{align*}
	By the Poincar\'e Lemma (see Lemma 10.68 in \cite{GiaMarti})  there exist  $B_\gamma\in W^{1,\frac{2m}{2m-2-|\gamma|}}_{\loc}(B_{1/2}^{2m}, \R^{n\times n }\otimes \wedge^2 \R^{2m}) $ for $|\gamma|\leq m-2$ such that 
	\begin{align*}
	\delta B_\gamma &= [...]_\gamma
	\end{align*}
	Now we transform  $\hat{\varepsilon}^*=\sum_{|\gamma|\leq m-2}\varepsilon_\gamma^*$ and $\hat{B}=\sum_{|\gamma|\leq m-2}B_\gamma$ back. Then we have $\varepsilon\in W^{m+1,\frac{2m}{m-1},1}(B_{1/2}^{2m}, M(n))$ with
	\begin{align*}
	||\varepsilon||_{W^{m+1,\frac{2m}{m-1},1}(B^{2m}_{1/2})} + ||\varepsilon||_{L^\infty(B^{2m}_{1/2})}\leq c\sigma
	\end{align*}
	and
	\begin{align*}
	\varepsilon=\sum_{|\gamma|\leq m-2}\p^{\gamma}\varepsilon^*_\gamma\qquad\text{solves (\ref{epsilon dgl})}.
	\end{align*} 
	Further $B=\sum_{|\gamma|\leq m-2}\p^\gamma B_\gamma \in W^{2-m,2}_{\loc}(B_{1/2}^{2m}, \R^{n\times n }\otimes \wedge^2 \R^{2m}) $ with
	\begin{align*}
	\delta B &= \sum_{k=0}^{m-1} \Delta^k 	((id+\varepsilon)P )V_k -\sum_{k=0}^{m-2} d\Delta^k	((id+\varepsilon)P )w_k  + d\Delta^{m-1}	((id+\varepsilon)P )
	\end{align*}
	and 
	\begin{align*}
	\delta&\bigg[ \sum_{l=0}^{m-1}\Delta^l ((id+\varepsilon)P) \Delta^{m-l-1} du - \sum_{l=0}^{m-2} d\Delta^l((id+\varepsilon)P) \Delta^{m-l-1} u
	\\\nonumber
	&\qquad -\sum_{k=0}^{m-1}\sum_{l=0}^{k-1} \Delta^l ((id+\varepsilon)P )\Delta^{k-l-1} d\langle V_k, du\rangle
	\\\nonumber
	&\qquad +\sum_{k=0}^{m-1}\sum_{l=0}^{k-1} d\Delta^l ((id+\varepsilon)P)\Delta^{k-l-1} \langle V_k, du\rangle
	\\\nonumber
	&\qquad - \sum_{k=0}^{m-2}\sum_{l=0}^{k} \Delta^l ((id+\varepsilon)P) d\Delta^{k-l-1}\delta (w_k du) 
	\\
	&\qquad+  \sum_{k=0}^{m-2}\sum_{l=0}^{k-1} d\Delta^l ((id+\varepsilon)P ) \Delta^{k-l-1}\delta (w_k du)  -\langle B, d u\rangle \bigg]=0.
	\end{align*}
\subsection{Regularity}
	To show $ (iii) $ we  abbreviate the conservation law (\ref{conservation law}) 
	\begin{align}\label{conslaw continous proof}
	\Delta\left((id+\varepsilon)P\Delta^{m-1}u\right) +\delta C =0\qquad\text{on }B^{2m}_{1/2},
	\end{align}
	where $C\in W^{2-m,\frac{2m}{m+1},1}(B_{1/2}^{2m})$. Since $\varepsilon\in W^{m+1, \frac{2m}{m-1},1}\cap L^\infty(B^{2m}_{1/2}), ~ P\in W^{m,2}\cap L^\infty(B^{2m}_{1/2})$ and $\Delta^{m-1}u_r \in W^{2-m,2}(B^{2m}_{1/2})$ we have 
	\begin{align}\label{harm eqn cont}
	(id+\varepsilon)P\Delta^{m-1}u\in  W^{2-m,2}(B^{2m}_{1/2}).
	\end{align}
	Set $f= (id+\varepsilon)P\Delta^{m-1}u$. Then
	\begin{align*}
	-\Delta f=\delta C \qquad\text{on }B_{1/2}^{2m}.
	\end{align*}	
	By Theorem 6.2 in \cite{LongDiss} we get $f\in W^{3-m, \frac{2m}{m+1},1}(B_{\lambda})$ on a smaller ball with radius $0<\lambda<1/2$. Since $(id+\varepsilon)P$ is invertible we rewrite (\ref{harm eqn cont}) 
	\begin{align*}
	\Delta^{m-1}u= \left[ (id+\varepsilon)P\right]^{-1}f
	\end{align*}
	and $\Delta^{m-1}u\in W^{3-m,\frac{2m}{m+1},1}(B^{2m}_{\lambda})$. But this means $u\in W^{m+1,\frac{2m}{m+1},1}(B^{2m}_\lambda)$ and $ W^{m+1,\frac{2m}{m+1},1}(B^{2m}_{\lambda}) \hookrightarrow C^0(B^{2m}_{\lambda})$  (see Theorem 2.3 in \cite{LongGas}). 

	Up until now we have assumed that $\sigma $ is arbitrarily small so that it satisfies the assumptions of Theorem \ref{Uhlenbeck thm} and the fixed point argument. A priori this is not true for components $V_k,w_k$ of a  system of the form (\ref{m poly system}). However any solution $u$ is continuous. To see this we rescale $u$ (see \cite{LongDiss} for a detailed proof). Let $x_0\in B^{2m}$ and $r>0$ small enough so that $u_r:B^{2m} \rightarrow \R^n,~u_r(x):=u(x_0+rx)$ is a solution of (\ref{m poly system}) on $B^{2m}$ with corresponding rescaled components $V_{k,r}$ and $w_{k,r}$,
	\begin{align*}
	\sigma_r &:= \sum_{k=0}^{m-2} ||w_{k,r}||_{W^{2k+2-m, 2}(B^{2m})} +\sum_{k=1}^{m-1}||V_{k,r}||_{W^{2k+1-m, 2}(B^{2m})} 
	\\
	&\qquad+ ||\eta_r||_{W^{2-m,2}(B^{2m})} + ||F_r||_{W^{2-m, \frac{2m}{m+1},1}(B^{2m})},
	\end{align*}
	$\sigma_r<\sigma_0$ and $B^{2m}_r(x_0)\subset B^{2m}$. By the above we have $u_r\in C^{0}(B^{2m}_{\lambda })$ which is the same as $u\in C^0(B^{2m}_{r\lambda}(x_0))$. A simple covering argument yields $u\in C^0(B^{2m})$.

\qed


\end{document}